\documentclass[11pt, reqno]{amsart}

\usepackage[top=3.75cm, bottom=3cm, left=3cm, right=3cm]{geometry}
\usepackage{amsmath}
\usepackage{amsthm}
\usepackage{amsfonts}
\usepackage{amssymb}
\usepackage{eucal}
\usepackage{bm}
\usepackage{hyperref}
\usepackage{color}
\usepackage[all,cmtip]{xy}
\usepackage{enumitem}
\usepackage{tikz}  
\usepackage{graphicx}
\usepackage{scalerel}

\numberwithin{equation}{section}

\theoremstyle{plain}
\newtheorem{theorem}{Theorem}[section]
\newtheorem{lemma}[theorem]{Lemma}

\theoremstyle{definition}
\newtheorem{definition}[theorem]{Definition}
\newtheorem{example}[theorem]{Example}
\newtheorem{remark}[theorem]{Remark}

\setlist[itemize]{leftmargin=*, itemsep={2pt}}
\setlist[enumerate]{leftmargin=*, itemsep={2pt}}

\makeatletter
\newcommand*{\da@rightarrow}{\mathchar"0\hexnumber@\symAMSa 4B }
\newcommand*{\da@leftarrow}{\mathchar"0\hexnumber@\symAMSa 4C }
\newcommand*{\xdashrightarrow}[2][]{%
  \mathrel{%
    \mathpalette{\da@xarrow{#1}{#2}{}\da@rightarrow{\,}{}}{}%
  }%
}
\newcommand{\xdashleftarrow}[2][]{%
  \mathrel{%
    \mathpalette{\da@xarrow{#1}{#2}\da@leftarrow{}{}{\,}}{}%
  }%
}
\newcommand*{\da@xarrow}[7]{%
  \sbox0{$\ifx#7\scriptstyle\scriptscriptstyle\else\scriptstyle\fi#5#1#6\m@th$}%
  \sbox2{$\ifx#7\scriptstyle\scriptscriptstyle\else\scriptstyle\fi#5#2#6\m@th$}%
  \sbox4{$#7\dabar@\m@th$}%
  \dimen@=\wd0 %
  \ifdim\wd2 >\dimen@
    \dimen@=\wd2 %
  \fi
  \count@=2 %
  \def\da@bars{\dabar@\dabar@}%
  \@whiledim\count@\wd4<\dimen@\do{%
    \advance\count@\@ne
    \expandafter\def\expandafter\da@bars\expandafter{%
      \da@bars
      \dabar@ 
    }%
  }%
  \mathrel{#3}%
  \mathrel{%
    \mathop{\da@bars}\limits
    \ifx\\#1\\%
    \else
      _{\copy0}%
    \fi
    \ifx\\#2\\%
    \else
      ^{\copy2}%
    \fi
  }%
  \mathrel{#4}%
}
\makeatother

\newcommand{\bijartop}[1][]{%
 \ar[#1]
 \ar@<0.7ex>@{}[#1]|-*=0[@]{\sim}} 
 
 \newcommand{\bijarbottom}[1][]{%
 \ar[#1]
 \ar@<-0.95ex>@{}[#1]|-*=0[@]{\sim}} 

\newcommand{\st}{\mid} 
\newcommand{\set}[1]{\left\{ \, #1 \, \right\}}

\newcommand{\Db}{\mathrm{D^b}}

\newcommand{\hpd}{{\natural}}
\newcommand{\svee}{\scriptscriptstyle\vee}

\DeclareMathOperator{\Hom}{Hom}

\newcommand{\Spec}{\mathrm{Spec}}

\newcommand{\Gr}{\mathrm{Gr}}

\newcommand{\Spin}{\mathrm{Spin}}
\newcommand{\SO}{\mathrm{SO}}

\newcommand{\cov}[1]{{#1}_{\textrm{cov}}}

\newcommand{\eps}{\varepsilon}

\DeclareMathOperator{\Sym}{Sym}

\newcommand{\vQ}{Q^{\svee}}

\DeclareMathOperator{\Cl}{\mathsf{Cliff}}

\newcommand{\vV}{V^{\svee}}

\newcommand{\tV}{\tilde{V}}

\newcommand{\CS}{\cC\cS}

\newcommand{\tf}{\tilde{f}} 

\newcommand{\fZ}{\mathfrak{Z}}

\newcommand{\End}{\operatorname{\mathrm{End}}}

\newcommand{\pr}{\mathrm{pr}}

\newcommand{\cO}{\mathcal{O}}
\newcommand{\cA}{\mathcal{A}}
\newcommand{\cB}{\mathcal{B}}
\newcommand{\cC}{\mathcal{C}}

\newcommand{\cE}{\mathcal{E}}
\newcommand{\cF}{\mathcal{F}}

\newcommand{\cH}{\mathbf{H}}

\newcommand{\cQ}{\scalerel*{\mathcal{Q}}{Q}}
\newcommand{\cR}{\mathcal{R}}
\newcommand{\cS}{\mathcal{S}}

\newcommand{\cW}{\mathcal{W}}

\newcommand{\bH}{\mathbf{H}}

\newcommand{\bZ}{\mathbf{Z}}
\newcommand{\bP}{\mathbf{P}}

\newcommand{\bk}{\mathbf{k}}

\newcommand{\sS}{\mathsf{S}}

\begin{document}

\title{Homological projective duality for quadrics}

\author{Alexander Kuznetsov}
\address{{\sloppy
\parbox{0.9\textwidth}{
Steklov Mathematical Institute of Russian Academy of Sciences,\\
8 Gubkin str., Moscow 119991 Russia
\\[5pt]
National Research University Higher School of Economics, Moscow, Russia
}\bigskip}}
\email{akuznet@mi-ras.ru \medskip}

\author{Alexander Perry}
\address{Department of Mathematics, Columbia University, New York, NY 10027 \smallskip}
\email{aperry@math.columbia.edu}

\thanks{A.K. was partially supported by the HSE University Basic Research Program, Russian Academic Excellence Project ``5-100''. 
A.P. was partially supported by NSF postdoctoral fellowship DMS-1606460, 
NSF grant DMS-1902060, and the Institute for Advanced Study. }

\begin{abstract}
We show that over an algebraically closed field of characteristic not equal to~2, 
homological projective duality for smooth quadric hypersurfaces 
and for double covers of projective spaces branched over smooth quadric hypersurfaces 
is a combination of two operations:
one interchanges a quadric hypersurface with its classical projective dual
and the other interchanges a quadric hypersurface with the double cover branched along it.
\end{abstract}

\maketitle


\section{Introduction} 
\label{section-intro}

The theory of homological projective duality (HPD) was introduced in~\cite{kuznetsov-hpd} 
as a way to describe derived categories of linear sections of interesting algebraic varieties.
Since then it was generalized to the noncommutative situation~\cite{NCHPD} and significantly developed in~\cite{joins}.
See~\cite{kuznetsov2014semiorthogonal} and~\cite{thomas2015notes} for surveys of the subject.

Roughly, HPD says that the derived categories of linear sections of a smooth projective variety 
mapping to a projective space $X \to \bP(V)$ are governed by a single (noncommutative) algebraic variety~$X^\hpd \to \bP(\vV)$ 
over the dual projective space, called the HP dual of~$X$.
The computation of~$X^{\hpd}$ thus becomes the main step in understanding these categories.

It is no surprise then that the computation of HP duals is quite hard in general. 
There are not so many examples for which an explicit geometric description 
of the HP dual is known;  
most are listed in~\cite{kuznetsov2014semiorthogonal} 
(see also~\cite{rennemo2015homological}, \cite{rennemo2016hori}, and~\cite[\S\S C--D]{kuznetsov2017sextic} for examples that appeared later).
One of the most basic examples, HPD for smooth quadrics, was stated in~\cite[Theorem~5.2]{kuznetsov2014semiorthogonal} without proof. 
The goal of this paper is to supply a proof. 

To give a precise statement, which we call \emph{quadratic HPD}, recall that
HPD deals with varieties $f \colon X \to \bP(V)$ that 
are equipped with a Lefschetz structure, which is a special type of semiorthogonal decomposition 
of the bounded derived category of coherent sheaves $\Db(X)$ (see~\S\ref{subsection-HPD}). 
In Theorem~\ref{main-theorem} below, both $f \colon Q \to \bP(V)$ and $f^{\hpd} \colon Q^{\hpd} \to \bP(V)$ 
are equipped with natural Lefschetz structures defined in terms of spinor bundles (see Lemma~\ref{lemma-smooth-Q-lc}). 

We work over an algebraically closed field $\bk$ of characteristic not equal to~$2$. 
Recall that the classical projective dual of a smooth quadric hypersurface $Q \subset \bP(V)$ is itself 
a smooth quadric hypersurface $Q^{\svee} \subset \bP(\vV)$. 
The HP dual of $Q$ is more subtle: 

\begin{theorem}
\label{main-theorem}
Let~\mbox{$f \colon Q \to \bP(V)$} be either the embedding of a  
smooth irreducible quadric hypersurface 
or a double cover branched along a smooth quadric hypersurface, 
equipped with its natural Lefschetz structure as in Lemma~\textup{\ref{lemma-smooth-Q-lc}}. 
The homological projective dual~\mbox{$f^\hpd \colon Q^\hpd \to \bP(\vV)$} of~\mbox{$f \colon Q \to \bP(V)$} is given as follows: 
\begin{enumerate}
\item 
\label{HPD-quadrics-1}
If $f$ is an embedding and $\dim(Q)$ is even, then $Q^\hpd = Q^{\svee}$ is the classical projective dual of~$Q$
and~$f^\hpd \colon Q^{\natural} \to \bP(\vV)$ is its natural embedding. 
\item 
\label{HPD-quadrics-d-odd} 
If $f$ is an embedding and $\dim(Q)$ is odd, then $f^\hpd \colon Q^{\natural} \to \bP(\vV)$ is the double cover 
branched along the classical projective dual of $Q$. 
\item 
\label{HPD-quadrics-dc-even}
If $f$ is a double covering and $\dim(Q)$ is even, then $f^\hpd \colon Q^{\natural} \to \bP(\vV)$ is the classical projective 
dual of the branch locus of $f$. 
\item 
\label{HPD-quadrics-dc-odd} 
If $f$ is a double covering and $\dim(Q)$ is odd, then $f^\hpd \colon Q^{\natural} \to \bP(\vV)$ is the double cover 
branched along the classical projective dual of the branch locus of $f$. 
\end{enumerate} 
In all cases $Q^\natural$ is considered with the Lefschetz structure defined in Lemma~\textup{\ref{lemma-smooth-Q-lc}}. 
\end{theorem} 

An important ingredient in HPD is the {HPD kernel}, which is an object 
\begin{equation*}
\cE \in \Db((X \times X^\hpd) \times_{\bP(V) \times \bP(\vV)} \bH),
\end{equation*}
where $\bH \subset \bP(V) \times \bP(\vV)$ is the incidence divisor, that provides all of the important functors.
At the end of the paper (Remark~\ref{remark:hpd-kernels}) we describe the HPD kernels for quadratic HPD.

Theorem~\ref{main-theorem} is a key ingredient in 
\cite{categorical-cones}, where using ``categorical cones'' we 
bootstrap to a description of HPD even when $Q$ is not smooth 
and its image does not span $\bP(V)$.  
As shown in \cite{categorical-cones}, this leads to a powerful 
description of the derived categories of quadratic sections of varieties, 
which among other things proves the duality conjecture for Gushel--Mukai varieties 
from~\cite{kuznetsov2016perry}.

In general the HP dual of a Lefschetz variety is noncommutative 
(i.e. is a suitably enhanced triangulated category, see the discussion before Definition~\ref{definition-lef-equiv}),
but quadratic HPD turns out to be a purely commutative statement.
Thanks to this, in the present paper we do not need the noncommutative setup of~\cite{NCHPD}.
However, we need some results on HPD that were proved in~\cite{NCHPD} and~\cite{categorical-cones};
to reformulate the corresponding statements in the commutative setup of this paper one just has to replace 
Lefschetz categories by derived categories of Lefschetz varieties.

The paper is organized as follows. 
In \S\ref{section-HPD-quadrics} we briefly review the theory of 
HPD and describe the Lefschetz structure of a quadric. 
Then in \S\ref{section-proof} we prove Theorem~\ref{main-theorem}. 

All functors (pullbacks, pushforwards, tensor products) in this paper are derived,
and the base field $\bk$ is an algebraically closed field of characteristic not equal to~$2$.

\subsection*{Acknowledgements} 
We would like to thank the referee for useful comments.

\section{HPD and the Lefschetz structure of quadrics}
\label{section-HPD-quadrics} 

In this section, we begin by describing the framework of HPD. 
Then we explain how quadrics can naturally be regarded as Lefschetz varieties, 
and hence can be considered as objects of this theory. 

\subsection{Homological projective duality}
\label{subsection-HPD} 
We recall the basics of HPD in the form presented in~\cite{NCHPD} and~\cite{joins},
but to simplify the exposition we focus on the purely smooth and proper commutative setting, 
which is sufficient for our purposes. 

Let $X$ be a smooth proper variety over $\bk$, and let $f \colon X \to \bP(V)$ 
be a morphism to a projective space. 
We denote by $\Db(X)$ the bounded derived category of coherent sheaves on~$X$.
A \emph{Lefschetz center} of $\Db(X)$ is an admissible subcategory $\cA_0 \subset \Db(X)$ such that there are semiorthogonal decompositions
\begin{equation}
\label{eq:lefschetz-decompositions}
\begin{aligned}
\Db(X) &= \langle \cA_0, \cA_1(1), \dots, \cA_{m-1}(m-1) \rangle,\\
\Db(X) &= \langle \cA_{1-m}(1-m), \dots, \cA_{-1}(-1), \cA_0 \rangle,
\end{aligned}
\end{equation} 
called respectively the \emph{right} and \emph{left Lefschetz decomposition} of $\Db(X)$,
whose components, called the \emph{Lefschetz components} of $\Db(X)$, form two chains of admissible subcategories 
\begin{equation*}
0 \subset \cA_{1-m} \subset \dots \subset \cA_{-1} \subset \cA_0 \supset \cA_1 \supset \dots \supset \cA_{m-1}  \supset 0.
\end{equation*}
Here, $\cA_i(i)$ denotes the image of $\cA_i$ under the autoequivalence of $\Db(X)$ given by tensoring with $f^* \cO_{\bP(V)}(i)$. 
We call $f \colon X \to \bP(V)$ a \emph{Lefschetz variety} if it is equipped with a Lefschetz center $\cA_0 \subset \Db(X)$.
By~\cite[Lemma~6.3]{NCHPD} the existence of one of the decompositions~\eqref{eq:lefschetz-decompositions} 
implies the existence of the other,  
the components~$\cA_i$ are completely determined by the Lefschetz center $\cA_0$, 
and~$\cA_{m-1} \ne 0$ if and only if $\cA_{1-m} \ne 0$.
The minimal $m$ with this property is called the \emph{length} of the Lefschetz variety $X$. 
We say $X$ is \emph{moderate} if $m < \dim(V)$ (see~\cite[Remark~2.12]{joins} for a discussion of this notion).

Let $\bH \subset \bP(V) \times \bP(\vV)$ be the natural incidence divisor. 
Let $\bH(X) := X \times_{\bP(V)} \bH$,
so that we have a commutative diagram 
\begin{equation*} 
\xymatrix{
\bH(X) \ar[r]^-{\delta} \ar[d] & X \times \bP(\vV) \ar[d] \ar[r] & X \ar[d]^{f} \\ 
\bH \ar[r] \ar[dr] & \bP(V) \times \bP(\vV) \ar[r] \ar[d] & \bP(V) \\ 
& \bP(\vV) &
}
\end{equation*} 
with cartesian squares. 
We denote by 
\begin{equation*}
\pi_X \colon \bH(X) \to X
\qquad\text{and}\qquad
h_X \colon \bH(X) \to \bP(\vV)
\end{equation*}
the natural projections.

The \emph{HPD category} of a Lefschetz variety $f \colon X \to \bP(V)$ is 
the triangulated subcategory of~$\Db(\bH(X))$ defined by 
\begin{equation*}
\Db(X)^\hpd := \set{ \cF \in \Db(\bH(X)) \st  \delta_*(\cF) \in \cA_0 \boxtimes \Db(\bP(\vV)) },
\end{equation*}
Here, $\cA_0 \boxtimes \Db(\bP(\vV))$ denotes the triangulated subcategory of~$\Db(X \times \bP(\vV))$ 
generated by box tensor products of objects in each factor.
The HPD category can alternatively be characterized by the $\bP(\vV)$-linear semiorthogonal 
decomposition 
\begin{equation}
\label{eq:hpd-sod}
\Db(\bH(X)) = \Big\langle \Db(X)^\hpd, \delta^*(\cA_1(1) \boxtimes \Db(\bP(\vV))), \dots, \delta^*(\cA_{m-1}(m-1) \boxtimes \Db(\bP(\vV))) \Big\rangle,
\end{equation} 
where the $\cA_i$ are the Lefschetz components of $\Db(X)$. 

Morally, the HP dual variety of $X$ is a variety whose derived category is equivalent to the HPD category of $X$.
A priori, the HP dual of $X$ may not exist as an algebraic variety.  
However, if $X$ is moderate
then the HP dual 
always exists as a noncommutative Lefschetz variety \cite[Theorem~8.7(1)]{NCHPD}. 
More precisely, $\Db(X)^{\hpd}$ has the structure of a $\bP(\vV)$-linear category 
in the sense of \cite[\S2]{NCHPD}. 
The notion of a Lefschetz center extends to such categories, and 
$\Db(X)^{\hpd}$ has a canonical Lefschetz center given by 
\begin{equation}
\label{cA0-hpd}
\cA_0^\hpd  = \gamma^*\pi_X^*(\cA_0) \subset \Db(X)^\hpd
\end{equation} 
where $\gamma^*$ denotes the left adjoint of the 
inclusion $\gamma \colon \Db(X)^{\hpd} \to \Db(\bH(X))$. 
This gives $\Db(X)^{\hpd}$ the structure of a Lefschetz category over $\bP(\vV)$
and allows us to make the following definition. 

\begin{definition}
\label{definition-lef-equiv}
A Lefschetz variety $f^\hpd \colon X^\hpd \to \bP(\vV)$ is \emph{HP dual} to a moderate Lefschetz variety~$f \colon X \to \bP(V)$
if there is a Fourier--Mukai kernel 
\begin{equation*}
\cE \in \Db(\bH(X) \times_{\bP(\vV)} X^\hpd), 
\end{equation*}
called the \emph{HPD kernel}, such that 
the corresponding Fourier--Mukai functor 
\begin{equation*}
\Phi_\cE \colon \Db(X^\hpd) \to \Db(\bH(X))
\end{equation*}
induces a \emph{Lefschetz equivalence} $\Db(X^\hpd) \simeq \Db(X)^\hpd$, i.e. an equivalence that 
identifies the Lefschetz centers on each side.
\end{definition} 

The definition of HPD can be conveniently reformulated as follows: 
if $\cA_0$ and $\cB_0$ are the Lefschetz centers of $X$ and $X^\hpd$, then
\begin{align}
\label{eq:hpd-prop-1}
&\Phi_\cE \colon \Db(X^\hpd) \xrightarrow{\ \sim \ } \Db(X)^\hpd \subset \Db(\bH(X)), 
\qquad\text{and}\qquad \\
\label{eq:hpd-prop-2}
&\Phi_\cE^*(\pi_X^* (\cA_0)) = \cB_0 \subset \Db(X^\hpd), 
\end{align} 
where $\Phi_\cE^*$ is the left adjoint functor of~$\Phi_\cE$.
Indeed, by~\eqref{eq:hpd-prop-1} the functor $\Phi_{\cE}$ can be written as~$\Phi_{\cE} = \gamma \circ \phi_{\cE}$ 
where $\phi_{\cE} \colon \Db(X^\hpd) \to \Db(X)^{\hpd}$ is an equivalence and $\gamma \colon \Db(X)^{\hpd} \to \Db(\bH(X))$ is the inclusion. 
Thus by the definition~\eqref{cA0-hpd} of the Lefschetz center $\cA_0^\hpd \subset \Db(X^{\hpd})$, 
condition~\eqref{eq:hpd-prop-2} can be rewritten as $\phi_{\cE}^*( \cA_0^{\hpd}) = \cB_0$. 
Since $\phi_{\cE}^* \colon \Db(X)^{\hpd} \to \Db(X^\hpd)$ is inverse to the equivalence~$\phi_{\cE}$, 
this shows that $\Phi_{\cE}$ identifies the Lefschetz centers of $\Db(X^\hpd)$ and~$\Db(X)^\hpd$.

To finish this brief introduction, we recall two important properties. 
First, HPD is really a duality: if~${X^\natural} \to \bP(\vV)$ is the HP dual variety of 
a smooth proper moderate Lefschetz variety~$X \to \bP(V)$
with its natural Lefschetz structure, then the HP dual variety of~${X^\natural}$ is~$X$ 
(see~\cite[Theorem 7.3]{kuznetsov-hpd} or~\cite[Theorem 8.9]{NCHPD}). 
Second, there is a tight connection between HPD and classical projective duality.
For instance, if the map $f \colon X \to \bP(V)$ is an embedding then the classical projective dual $X^{\svee} \subset \bP(\vV)$
coincides with the set of critical values of the map~\mbox{$X^\hpd \to \bP(\vV)$} 
from the HP dual variety~\cite[Theorem 7.9]{kuznetsov-hpd}.

\subsection{Spinor bundles and the Lefschetz structure of quadrics}

Let $Q$ be a smooth quadric, i.e. an integral scheme over $\bk$ 
which admits a closed immersion into a projective space as a quadric hypersurface. 
We denote by $\cO_Q(1)$ the restriction of the line bundle $\cO(1)$ from this ambient space.
The main result of this paper is a description of the HP dual of $Q$. 
To make sense of this, we need to specify the structure of a Lefschetz 
variety on $Q$, i.e. a morphism to a projective space and a Lefschetz 
center of $\Db(Q)$. 

First, we specify the class of morphisms that we consider.

\begin{definition}
We say a morphism $f \colon Q \to \bP(V)$ is \emph{standard} if there is an isomorphism
\begin{equation*}
f^*\cO_{\bP(V)}(1) \cong \cO_Q(1).
\end{equation*}
We call $f$ \emph{non-degenerate} if its image is not contained in a hyperplane of~$\bP(V)$.
\end{definition} 

In this paper we will only be concerned with non-degenerate standard maps of smooth quadrics; 
see~\cite[\S5]{categorical-cones} for results about degenerate maps, which are obtained using 
the non-degenerate case as a starting point.
If $f$ is standard and non-degenerate, then it is either a divisorial embedding or a double covering. 
Indeed, by definition any standard morphism $f$ is the composition of an embedding $Q \hookrightarrow \bP^n$ as a quadric hypersurface 
followed by a linear embedding $\bP^n \hookrightarrow \bP(V)$ into a larger projective space, 
or by a linear projection $\bP^n \dashrightarrow \bP(V)$ to a smaller projective space.
In the first case the embedding must be an isomorphism
in order for $f$ to be non-degenerate, and 
in the second case the map $\bP^n \dashrightarrow \bP(V)$ must be linear projection from a point 
not on~$Q$ in order for~$f$ to be a regular morphism. 
In the second case, $f$ is then a double cover of $\bP(V)$ branched along a quadric hypersurface. 

The Lefschetz center of $Q$ will be defined in terms of spinor bundles. 
We follow~\cite{ottaviani} for our conventions on spinor bundles, 
and recall some of the key facts here (see~\cite[Theorem~2.8]{ottaviani}). 

Let $Q$ be a smooth quadric of even dimension $2d$, and write $H$ for the hyperplane 
class so that $\cO(H) = \cO_Q(1)$. 
Let $\Spin(Q)$ be the universal covering of the special orthogonal group~$\SO(Q)$ associated with the quadric $Q$.
Then $Q$ carries a pair of~$\Spin(Q)$-equivariant
vector bundles $\cS_+$ and $\cS_-$ of rank $2^{d-1}$, called the \emph{spinor bundles}. 

\begin{example}
If $d = 1$ then $Q \cong \bP^1 \times \bP^1$ and~$\cS_+ = \cO(-1, 0)$ and~$\cS_- = \cO(0,-1)$.
If $d = 2$ then $Q \cong \Gr(2,4)$ and the bundles $\cS_+$ and $\cS_-$ are the two tautological subbundles 
of rank~2.
\end{example}

If $d \geq 2$ then $\cS_+$ and $\cS_-$ both have determinant~$\cO_Q(-2^{d-2})$.  
Denoting by $\sS_\pm$ the $2^d$-dimensional half-spinor representations of $\Spin(Q)$, 
there are canonical exact sequences
\begin{equation}
\label{eq:spinor-sequences-even}
0 \to \cS_+ \to \sS_+ \otimes \cO_Q \to \cS_-(H) \to 0,
\qquad 
0 \to \cS_- \to \sS_- \otimes \cO_Q \to \cS_+(H) \to 0.
\end{equation}
Moreover, if $f \colon Q \to \bP(V)$ is an embedding of $Q$ as a quadric hypersurface
the pushforwards of the spinor bundles have the following standard resolutions
\begin{equation}
\label{eq:spinor-sequence-pn}
\begin{aligned}
0 \to \sS_- \otimes \cO_{\bP(V)}(-H) \to \sS_+ \otimes \cO_{\bP(V)} \to f_*(\cS_-(H)) \to 0,\\
0 \to \sS_+ \otimes \cO_{\bP(V)}(-H) \to \sS_- \otimes \cO_{\bP(V)} \to f_*(\cS_+(H)) \to 0\hphantom{,}
\end{aligned}
\end{equation} 
(see, e.g., \cite[Example 3.4]{KF}, but note that~\cite{KF} uses a different convention for spinor bundles).
Another nice property of spinor bundles is their self-duality up to a twist:
\begin{equation}
\label{eq:spinor-duality-even}
\cS_\pm(H) \cong 
\begin{cases}
\cS_\pm^{\svee}, & \text{if $d$ is even},\\
\cS_\mp^{\svee}, & \text{if $d$ is odd}.
\end{cases}
\end{equation} 

Similarly, if $Q$ is a smooth quadric of odd dimension $2d - 1$, it carries one spinor bundle $\cS$ of rank $2^{d-1}$
that fits into an exact sequence
\begin{equation}
\label{eq:spinor-sequences-odd}
0 \to \cS \to \sS \otimes \cO_Q \to \cS(H) \to 0,
\end{equation}
where $\sS$ is the spinor representation of $\Spin(Q)$, and such that
\begin{equation}
\label{eq:spinor-duality-odd}
\cS(H) \cong \cS^{\svee}.
\end{equation} 
Moreover, if $Q$ is represented as a hyperplane section of a smooth quadric $Q'$ of even dimension, 
then $\cS$ is isomorphic to the restriction of either of the spinor bundles ${\cS_{\pm}}$ on $Q'$:
\begin{equation}
\label{eq:spinor-restriction}
\cS \cong \cS_\pm\vert_Q,
\end{equation} 
see~\cite[Theorem~1.4(i)]{ottaviani}.

In what follows, when $Q$ is a smooth quadric of arbitrary dimension, 
we will denote by $\cS$ a chosen spinor bundle on it --- 
the only one in the odd-dimensional case, or one of the two in the even-dimensional case.
With this convention, we have the following result.  

\begin{lemma}
\label{lemma-smooth-Q-lc}
Let $f \colon Q \to \bP(V)$ be a standard morphism of a smooth quadric $Q$. 
Let $\cS$ denote a spinor bundle on $Q$. 
Then $Q$ has the structure of a Lefschetz variety over $\bP(V)$ of length~$\dim(Q)$ with Lefschetz center
\begin{equation*}
\cQ_0 = \langle \cS , \cO \rangle 
\subset \Db(Q). 
\end{equation*}
Further, if $p \in \set{0,1}$ is the parity of $\dim(Q)$, i.e. $p = \dim(Q) \pmod 2$, 
then the nonzero Lefschetz components of $\Db(Q)$ are given by 
\begin{equation*}
\cQ_i = \begin{cases}
\langle \cS, \cO \rangle & \text{for $|i| \leq 1-p$} , \\ 
\langle \cO \rangle & \text{for $1-p < |i| \leq \dim(Q)-1$}.
\end{cases}
\end{equation*}
\end{lemma}

\begin{proof}
Kapranov's semiorthogonal decomposition of the derived category of a smooth $m$-dimensional quadric~\cite{kapranov} gives
\begin{align*}
\Db(Q) &= \langle \cS_+, \cS_-, \cO, \cO(1), \dots \cO(m-1) \rangle, && \text{if $m$ is even;}\\
\Db(Q) &= \langle \cS, \cO, \cO(1), \dots \cO(m-1) \rangle, && \text{if $m$ is odd.}
\end{align*}
Thus, in the odd-dimensional case we obtain the required Lefschetz structure.
In the even-dimensional case we use~\eqref{eq:spinor-sequences-even} to rewrite the decomposition in form
\begin{equation*}
\Db(Q) = \langle \cS_+, \cO, \cS_+(1), \cO(1), \dots \cO(m-1) \rangle,
\end{equation*}
(or similarly with $\cS_+$ replaced by $\cS_-$) and also obtain the required Lefschetz structure. 
\end{proof}

\begin{remark} 
\label{remark-smooth-Q-lc}
Let $f \colon Q \to \bP(V)$ be a standard morphism of a smooth quadric $Q$. 
Then we always regard $Q$ as a Lefschetz variety using the 
center from Lemma~\ref{lemma-smooth-Q-lc}. 
If $\dim(Q) = 2d$ is even there are two spinor bundles~$\cS_+$ and~$\cS_-$, 
so there is an apparent choice involved in the Lefschetz structure of~$Q$. 
However, there exists an (noncanonical) automorphism $a$ of~$Q$ over~$\bP(V)$ such that $a^*(\cS_\pm) \simeq \cS_\mp$  
(corresponding to the automorphism of the Dynkin diagram of type~$\mathrm{D}_{d+1}$). 
The resulting autoequivalence~$a^*$ of~$\Db(Q)$ identifies the Lefschetz center of Lemma~\ref{lemma-smooth-Q-lc} 
defined by~$\cS = \cS_+$ with that defined by~$\cS = \cS_-$. 
Hence, if~$\dim(Q)$ is even, the structure of~$Q$ as a Lefschetz variety over~$\bP(V)$ is still uniquely determined, 
up to noncanonical equivalence. 
\end{remark}

\begin{remark}
\label{remark:other-spinor-center}
The Lefschetz center $\cQ_0$ of $\Db(Q)$ can be also written as 
\begin{equation*}
\cQ_0 = \langle \cO, {\cS'}^{\svee}  \rangle
\end{equation*}
where $\cS' = \cS$ if $\dim(Q)$ is not divisible by 4, and the other spinor bundle otherwise. 
This follows from the exact sequences~\eqref{eq:spinor-sequences-even} and~\eqref{eq:spinor-sequences-odd}
and the dualities~\eqref{eq:spinor-duality-even} and~\eqref{eq:spinor-duality-odd}.
\end{remark}

\section{Proof of Theorem~\ref{main-theorem}}
\label{section-proof} 

In this section, we prove Theorem~\ref{main-theorem}. 
We divide the proof into a number of steps, which we overview here. 
Most of the proof concerns case~\eqref{HPD-quadrics-1}.   
In this case~$\dim(Q) = 2d$, $f \colon Q \to \bP(V)$ is a divisorial embedding, 
and we aim to prove that the HP dual is given by $Q^{\svee} \subset \bP(\vV)$, 
the classical projective dual of~$Q$.
During the proof we actively use the machinery of Clifford algebras; 
we suggest~\cite{kuznetsov08quadrics} as a general reference for this subject.

We consider the universal hyperplane section~$\bH(Q) \subset Q \times \bP(\vV)$ as a family of quadrics 
\begin{equation}
\label{eq:hq-family}
h_Q \colon \bH(Q) \to \bP(V^{\svee}) 
\end{equation}
of dimension~$2d - 1$.
\begin{itemize}
\item 
In~\S\ref{step-1} we introduce a sheaf $\Cl_0(\cW)$ of Clifford algebras on $\bP(\vV)$ and use the fibration~$h_Q$
to isolate a semiorthogonal component of $\Db(\bH(Q))$ equivalent 
to the derived category~\mbox{$\Db(\bP(V^{\svee}),\Cl_0(\cW))$} of modules over~$\Cl_0(\cW)$. 
\item 
In~\S\ref{step-2} we use central reduction to rewrite $\Db(\bP(V^{\svee}),\Cl_0(\cW))$ 
as the $\bZ/2$-equivariant category $\Db(Z,\cR)^{\bZ/2}$ of the derived category of modules
over an Azumaya algebra~$\cR$ on the double covering~$Z$ of~$\bP(\vV)$ branched over the quadric $\vQ$.
\item 
In~\S\ref{step-3} we show that the Azumaya algebra~$\cR$ is Morita trivial, 
and thus obtain an identification $\Db(Z,\cR)^{\bZ/2} \simeq \Db(Z)^{\bZ/2}$ with the $\bZ/2$-equivariant derived category of~$Z$.
\item 
In~\S\ref{step-4} we decompose~$\Db(Z)^{\bZ/2}$ into two components:
$\Db(\vQ)$ and $\Db(\bP(\vV))$.

\item 
In~\S\ref{step-5} we rewrite in a simpler form the embedding functor of $\Db(\bP(\vV))$ into $\Db(\bH(Q))$. 

\item 
In~\S\ref{step-6} we check that the image of $\Db(\bP(\vV))$ together with the other components of 
the semiorthogonal decomposition discussed in~\S\ref{step-1} generate the ``Lefschetz part'' of~$\Db(\bH(Q))$, 
i.e. the subcategory of~$\Db(\bH(Q))$ generated by the components to the right of~$\Db(Q)^{\hpd}$ 
in the decomposition \eqref{eq:hpd-sod} for $X = Q$. 
This proves that the remaining component~$\Db(\vQ)$ is equivalent to the HPD category $\Db(Q)^\hpd$.
We check that this equivalence is given by a Fourier--Mukai functor with kernel $\cE$ on $\bH(Q) \times_{\bP(\vV)} \vQ$,
thus verifying part~\eqref{eq:hpd-prop-1} of the definition of HPD.
\item In \S\ref{step-7} we describe the kernel $\cE$ in terms of spinor bundles on $Q$ and $\vQ$.
\item 
In~\S\ref{step-9} we use this to check that 
the functor~$\Phi^* \circ \pi_Q^*$ takes the Lefschetz center of $Q$ to that of $\vQ$, so 
condition~\eqref{eq:hpd-prop-2} in the definition of the HP dual holds.
\end{itemize}

Finally, in \S\ref{step-10} we deduce the remaining cases~\eqref{HPD-quadrics-d-odd}--\eqref{HPD-quadrics-dc-odd} 
of Theorem~\ref{main-theorem} from case~\eqref{HPD-quadrics-1} by applying a result from~\cite{carocci2015homological}.

\bigskip

In \S\S\ref{step-1}--\ref{step-9}, we assume as above that $\dim(Q) = 2d$, $f \colon Q \to \bP(V)$ is a divisorial embedding,
and that the Lefschetz structure of $\Db(Q)$ is chosen so that $\cQ_0 = \langle \cS_+, \cO \rangle$. 
Let $H$ and $H'$ denote the hyperplane classes on $\bP(V)$ and $\bP(\vV)$,
as well as their pullbacks to $\bH(Q)$ and other varieties.

\subsection{A decomposition of $\Db(\bH(Q))$}
\label{step-1}

Consider the family of quadrics~\eqref{eq:hq-family}.
Recall that by definition $\bH(Q)$ is the zero locus of the tautological global section of the line bundle~$\cO(H + H')$ on $Q \times \bP(\vV)$.
Therefore, $\bH(Q)$ sits as a family of $(2d-1)$-dimensional quadrics 
in the projectivization of the rank $2d+1$ vector bundle on $\bP(\vV)$ 
\begin{equation*}
\cW := (h_{Q*}\cO_{\bH(Q)}(H))^{\svee} \cong 
\operatorname{coker}(\cO(-H') \to \vV \otimes \cO)^\vee \cong
\ker(V \otimes \cO \to \cO(H')) \cong \Omega_{\bP(\vV)}(H') , 
\end{equation*}
and is defined by the family of quadratic forms given by the composition
\begin{equation}
\label{eq:quadratic-forms}
\cO \to \Sym^2\vV \otimes \cO \to \Sym^2\cW^{\svee},
\end{equation}
where the first morphism is given by the equation of $Q$ and the second is the tautological surjection.
By~\cite[Theorem~4.2]{kuznetsov08quadrics} we have a semiorthogonal decomposition
\begin{multline}
\label{eq:sod-hq-1}
\Db(\cH(Q)) = \langle \Db(\bP(V^{\svee}),\Cl_0(\cW)) , \\ 
h_Q^*(\Db(\bP(V^{\svee})))(H), 
\dots, h_Q^*(\Db(\bP(V^{\svee})))((2d-1)H) \rangle,
\end{multline}
where $\Cl_0(\cW)$ is the sheaf of even parts of the universal Clifford algebra on $\bP(\vV)$ for this family of quadrics, 
and~$\Db(\bP(V^{\svee}),\Cl_0(\cW))$ is the bounded derived category of~$\Cl_0(\cW)$-modules on~$\bP(\vV)$.
The embedding of this category into $\Db(\bH(Q))$ is given by the functor
\begin{equation}
\label{eq:def-gamma}
\gamma \colon \Db(\bP(V^{\svee}),\Cl_0(\cW)) \to \Db(\cH(Q)),
\qquad 
\cF \mapsto h_Q^*\cF \otimes_{\Cl_0(\cW)} \CS,
\end{equation} 
where $\CS$ is the sheaf of $\Cl_0(\cW)$-modules on $\bH(Q)$ defined by the exact sequence
\begin{equation}
\label{eq:ce-res}
0 \to \cO(-H) \otimes \Cl_0(\cW) \to \cO \otimes \Cl_1(\cW) \to i_*\CS \to 0
\end{equation}
on $\bH \cong \bP_{\bP(\vV)}(\cW)$, where $i \colon \bH(Q) \hookrightarrow \bH$ is the natural embedding. 
Furthermore, $\Cl_1(\cW)$ is the pullback to $\bH$ of the sheaf of odd parts of the Clifford algebra,  
and the first morphism is induced by the Clifford multiplication. 

We call $\CS$ the \emph{Clifford spinor bundle}.
Note that the space $\bH$ is simply the universal hyperplane in $\bP(V)$.

\subsection{Central reduction}
\label{step-2}

Next, we use the argument of~\cite[\S3.6]{kuznetsov08quadrics} to describe the first component~$\Db(\bP(V^{\svee}),\Cl_0(\cW))$ 
of~\eqref{eq:sod-hq-1} in more detail.
We denote by $\Cl(V)$ the Clifford algebra of the quadric~$Q$,
and by $\Cl_0(V)$ and $\Cl_1(V)$ its even and odd parts.

The family of quadratic forms~\eqref{eq:quadratic-forms} is given by a morphism $\cO \to \Sym^2\cW^{\svee}$ from a trivial line bundle, 
hence the Clifford multiplication provides the sum $\Cl_0(\cW) \oplus \Cl_1(\cW)$ with an algebra structure. 
As a sheaf of $\cO$-modules it has rank $2^{2d+1}$ and can be written as 
\begin{equation*}
\Cl(\cW) = \Cl_0(\cW) \oplus \Cl_1(\cW) \cong \cO \oplus \cW \oplus \wedge^2\cW \oplus \dots \oplus \wedge^{2d+1}\cW \subset 
\Cl(V) \otimes \cO,
\end{equation*}
and is naturally a subalgebra in $\Cl(V) \otimes \cO$.
As explained in~\cite[\S3.6]{kuznetsov08quadrics}, the rank~2 subalgebra 
\begin{equation*}
\fZ = \fZ_0 \oplus \fZ_1 = \cO \oplus \wedge^{2d+1}\cW \subset \Cl(\cW)
\end{equation*}
is central (and moreover $\Cl(\cW)$ is the centralizer of $\fZ$ in $\Cl(V)$), and the morphism
\begin{equation*}
\zeta \colon Z = \Spec_{\bP(\vV)}(\fZ) \to \bP(\vV)
\end{equation*}
is the double covering branched along the projective dual quadric $Q^{\svee} \subset \bP(\vV)$.

Note that $Z$ is a smooth quadric of dimension~$2d + 1$.
We consider the $\bZ/2$-action on $Z$ generated by the involution of the double covering.
Note that it is induced by the natural~$\bZ/2$-grading of $\fZ$.
The sheaf of algebras $\Cl(\cW)$ is a module over $\fZ$, hence there is a sheaf of algebras $\cR$ of rank~$2^{2d}$ on $Z$ such that 
\begin{equation*}
\Cl(\cW) \cong \zeta_*\cR.
\end{equation*}
Furthermore, the direct sum decomposition $\Cl(\cW) = \Cl_0(\cW) \oplus \Cl_1(\cW)$ 
provides $\Cl(\cW)$ with the structure of a $\bZ/2$-graded $\fZ$-module, hence
provides $\cR$ with a~$\bZ/2$-equivariant structure.
By definition the invariant part of $\zeta_*\cR$ is
\begin{equation}
\label{eq:zeta-s-cr}
(\zeta_*\cR)^{\bZ/2} \cong \Cl_0(\cW),
\end{equation}
hence there is an equivalence of categories
\begin{equation}
\label{eq:functor-perf-z-cr-z2}
\phi \colon \Db(Z,\cR)^{\bZ/2} \xrightarrow{\, \sim \,} \Db(\bP(\vV),\Cl_0(\cW)),
\qquad 
\cF \mapsto (\zeta_*\cF)^{\bZ/2}, 
\end{equation}
between the $\bZ/2$-equivariant derived category of $\cR$-modules on $Z$ and the derived category of~$\Cl_0(\cW)$-modules on $\bP(\vV)$.

\subsection{Morita triviality of $\cR$}
\label{step-3}

The sheaf of algebras $\cR$ is Azumaya by~\cite[Proposition~3.15]{kuznetsov08quadrics}. 
We claim it is in fact Morita trivial.
Indeed, let $\sS_+$ and $\sS_-$ be the two {$2^d$-dimensional half-spinor modules for $\Cl_0(V)$
(which appeared earlier as the half-spinor representations of~$\Spin(Q)$).
Then the sum 
\begin{equation}
\label{eq:spinor-sum}
\sS = \sS_+ \oplus \sS_-
\end{equation} 
is naturally a $\Cl(V)$-module, and hence by restriction a $\Cl(\cW)$-module as well.
In particular, it is a $\fZ$-module, hence gives a vector bundle on $Z$.
Moreover, the action of $\fZ_0$ preserves the summands, and the action of $\fZ_{1}$ swaps them, so
thinking of the direct sum decomposition~\eqref{eq:spinor-sum} as a $\bZ/2$-grading, 
we see that this provides $\sS$ with the structure of a $\bZ/2$-equivariant~$\fZ$-module,
i.e. a $\bZ/2$-equivariant vector bundle on~$Z$.
Since~$\sS$ is also a $\Cl(\cW)$-module,
we see that there is an object $\cS_Z^{\svee}$ of $\Db(Z,\cR)^{\bZ/2}$ such that
\begin{equation}
\label{eq:zeta-szv-z2}
\zeta_*\cS_Z^{\svee} \cong (\sS_+ \oplus \sS_-) \otimes \cO_{\bP(\vV)}
\qquad\text{and}\qquad 
(\zeta_*\cS_Z^{\svee})^{\bZ/2} \cong \sS_+ \otimes \cO_{\bP(\vV)}.
\end{equation}
Actually, $\cS_Z^{\svee}$ is the dual spinor bundle (of rank $2^d$) on the smooth odd-dimensional quadric~$Z$ (of dimension~$2d+1$), 
hence the notation.
Indeed, this follows easily from the Kapranov's semiorthogonal decomposition of the quadric~$Z$ 
(see the proof of Lemma~\ref{lemma-smooth-Q-lc}) since by~\eqref{eq:zeta-szv-z2} 
the sheaf~$\cS_Z^{\svee}$ is semiorthogonal to $\cO_Z(1)$, $\cO_Z(2)$, \dots, $\cO_Z(2d+1)$.

Since the bundle $\cS_Z^{\svee}$ is an equivariant $\cR$-module on $Z$, 
we have a natural equivariant morphism $\cR \to \operatorname{\mathcal{E}\!\mathit{nd}}(\cS_Z^{\svee})$,
which is fiberwise injective because $\cR$ is an Azumaya algebra, and hence is an isomorphism 
because the ranks of the source and the target are both equal to $2^{2d}$.
Consequently, we have an equivalence
\begin{equation}
\label{eq:functor-perf-z-z2}
\mu \colon \Db(Z)^{\bZ/2} \xrightarrow{\, \sim \,} \Db(Z,\cR)^{\bZ/2},
\qquad 
\cF \mapsto \cF \otimes \cS_Z^{\svee}.
\end{equation} 

\subsection{Root stack decomposition}
\label{step-4}

Finally, the equivariant category $\Db(Z)^{\bZ/2}$ can be considered as the derived category 
of the quotient stack $[Z / (\bZ/2)]$, i.e. of the root stack of~$\bP(\vV)$ along $Q^{\svee}$, 
and consequently by~\cite[Theorem~4.1]{cyclic-covers} (see also~\cite{collins2010polishchuk} and~\cite[Theorem~1.6]{ishii2015ueda})
it has a semiorthogonal decomposition
\begin{equation}
\label{eq:sod-root}
\Db(Z)^{\bZ/2} = \langle \Db(Q^{\svee}), \Db(\bP(\vV)) \rangle,
\end{equation}
with the embedding functors given by 
\begin{alignat}{2}
\label{eq:alpha-q}
\alpha_Q &\colon \Db(Q^{\svee}) \to \Db(Z)^{\bZ/2},
&& \cF \mapsto j_*\cF \otimes \chi,
\intertext{where $j \colon Q^{\svee} \to Z$ is the embedding of the ramification divisor 
and $\chi$ is the nontrivial character of the group~$\bZ/2$, and by}
\label{eq:alpha-p}
\alpha_P &\colon \Db(\bP(\vV)) \to \Db(Z)^{\bZ/2}, \qquad 
 && \cF \mapsto \zeta^*\cF,
\end{alignat}
where $\zeta^*\cF$ is given the natural equivariant structure.

Combining~\eqref{eq:sod-hq-1} with \eqref{eq:functor-perf-z-cr-z2}, \eqref{eq:functor-perf-z-z2}, 
and \eqref{eq:sod-root}, we obtain a $\bP(\vV)$-linear 
semiorthogonal decomposition
\begin{multline}
\label{eq:sod-hq-2}
\Db(\bH(Q)) = \langle \Phi(\Db(Q^{\svee})), \Psi(\Db(\bP(\vV))) ,\\ 
h_Q^*(\Db(\bP(V^{\svee})))(H), \dots, h_Q^*(\Db(\bP(V^{\svee})))((2d-1)H) \rangle .
\end{multline}
The embedding functors $\Phi$ and $\Psi$ of the first two components are discussed below.

\subsection{Rewriting the functor $\Psi$} 
\label{step-5}

According to the construction in~\S\S\ref{step-1}--\ref{step-4} above,
the second component of~\eqref{eq:sod-hq-2} is embedded by the functor 
\begin{equation*}
\Psi = \gamma \circ \phi \circ \mu \circ \alpha_P \colon \Db(\bP(\vV)) \to \Db(\bH(Q)),
\end{equation*}  
where the factors are defined by~\eqref{eq:def-gamma}, \eqref{eq:functor-perf-z-cr-z2}, \eqref{eq:functor-perf-z-z2}, 
and~\eqref{eq:alpha-p}.
Note that each of the factors is a Fourier--Mukai functor, hence so is their composition $\Psi$.
Below we describe its kernel object.
We consider the commutative diagram 
\begin{equation}
\label{diagram-HQ-PW}
\vcenter{\xymatrix{
&& Z \ar[d]^\zeta
\\
\bH(Q) \ar[r]^-{i} \ar[d]_{{\pi_Q}} \ar@/^4ex/[rr]^-{h_Q} & 
\bH \ar[r]^-{h} \ar[d]^{{\pi}}  & 
\bP(\vV)  
\\ 
Q \ar[r]^{f} & 
\bP(V) 
}}
\end{equation} 
with cartesian square. 
The functor $\Psi$
is given by
\begin{align*}
\cF 
&\mapsto {h_Q^*}{\left( \zeta_*\big(\zeta^*\cF \otimes {\cS_Z^{\svee}}\big)^{\bZ/2} \right)} \otimes _{\Cl_0(\cW)} \CS \\
& \cong  {h_Q^*}{ \left( \cF \otimes (\zeta_*{\cS_Z^{\svee}})^{\bZ/2} \right)} \otimes _{\Cl_0(\cW)} \CS \\
& \cong  ({h_Q^*}\cF \otimes \sS_+) \otimes _{\Cl_0(\cW)} \CS , 
\end{align*}
where we used the equivariant projection formula 
for the first isomorphism, and~\eqref{eq:zeta-szv-z2} for the second. 
{This means that the Fourier--Mukai kernel for $\Psi$ is the object 
\begin{equation*}
\sS_+ \otimes _{\Cl_0(\cW)} \CS \in \Db(\bH(Q)).
\end{equation*}
To compute it we use the resolution~\eqref{eq:ce-res}, and obtain on~$\bH$ a distinguished triangle} 
\begin{multline}
\label{iPhicF-triangle}
\sS_+ \otimes _{\Cl_0(\cW)} (\cO(-H) \otimes \Cl_0(\cW)) \to
\sS_+ \otimes _{\Cl_0(\cW)} (\cO \otimes \Cl_1(\cW)) \\
\to i_*(\sS_+ \otimes _{\Cl_0(\cW)} \CS)
\end{multline} 
with the first map induced by the Clifford multiplication.
The first term is evidently isomorphic to~{$\sS_+ \otimes \cO(-H)$}.
For the second note that
\begin{equation}
\label{eq:spm-tensor-cl1}
\begin{aligned}
\sS_+ \otimes_{\Cl_0(\cW)} \Cl_1(\cW) & \cong 
{\sS_+ \otimes_{\Cl_0(V)} (\Cl_0(V) \otimes_{\Cl_0(\cW)} \Cl_1(\cW))} \\
& \cong \sS_+ \otimes_{\Cl_0(V)} \Cl_1(V) \\ 
& \cong \sS_-.  
\end{aligned}
\end{equation}
Here the first isomorphism is evident.
The second is induced by Clifford multiplication; its surjectivity is evident, 
and its injectivity follows from the fact that $\Cl_1(\cW)$ is locally projective over $\Cl_0(\cW)$, see~\cite[Lemma~3.8]{kuznetsov08quadrics}.
Finally,
the last isomorphism follows from the standard isomorphisms 
\begin{equation*}
\Cl_0(V) \cong \End(\sS_+) \oplus \End(\sS_-) 
\quad \text{and} \quad \Cl_1(V) \cong \Hom(\sS_+,\sS_-) \oplus \Hom(\sS_-,\sS_+) . 
\end{equation*}
Hence the second term in \eqref{iPhicF-triangle} is isomorphic to~$\sS_- \otimes \cO$.
Thus, we can rewrite~\eqref{iPhicF-triangle} as
\begin{equation}
\label{iPhicF-triangle-rewritten}
{\sS_+ \otimes \cO(-H) \to \sS_- \otimes \cO \to i_*(\sS_+ \otimes _{\Cl_0(\cW)} \CS)}
\end{equation}
with the first map induced by the Clifford multiplication.

On the other hand, on $\bP(V)$ we have exact sequences~\eqref{eq:spinor-sequence-pn}.
Pulling the second of them back via $\pi \colon \bH \to \bP(V)$ and
using the base change isomorphism for the square in diagram~\eqref{diagram-HQ-PW},
we deduce an isomorphism
\begin{equation*}
i_*(\sS_+ \otimes _{\Cl_0(\cW)} \CS) \cong \pi^*(f_*(\cS_+(H))) \cong i_*(\pi_Q^*(\cS_+(H))).
\end{equation*}
Since $i$ is a closed embedding and $\pi_Q^*(\cS_+(H))$ is a coherent sheaf, 
it follows that
\begin{equation*}
\sS_+ \otimes _{\Cl_0(\cW)} \CS \cong \pi_Q^*(\cS_+(H)).
\end{equation*} 
In summary, we conclude that the second component of the decomposition~\eqref{eq:sod-hq-2} 
is embedded via the functor 
\begin{equation}
\label{equation-embedding-Pv-HQ}
\Psi \colon \Db(\bP(\vV)) \to \Db(\bH(Q)), \quad \cF \mapsto {h_Q^*}\cF \otimes {\pi_Q^*}(\cS_+(H)). 
\end{equation} 

\subsection{An equivalence between $\Db(\vQ)$ and $\Db(Q)^\hpd$}
\label{step-6}

Next, we relate the functor~$\Psi$ to the decomposition~\eqref{eq:hpd-sod} of~$\bH(Q)$.
Note that this decomposition takes the form 
\begin{multline}
\label{eq:hpd-sod-q}
\Db(\cH(Q)) = \langle \Db(Q)^\hpd, \delta^*(\cQ_1(H) \boxtimes \Db(\bP(\vV))), \\
\delta^*(\cQ_2(2H)  \boxtimes \Db(\bP(V^{\svee}))), \dots, \delta^*(\cQ_{2d-1}((2d-1)H)  \boxtimes \Db(\bP(V^{\svee}))) \rangle,
\end{multline}
where $\cQ_1 = \langle {\cS_+}, \cO \rangle$ and $\cQ_2 = \dots = \cQ_{2d-1} = \langle \cO \rangle$ are 
the Lefschetz components of~$\Db(Q)$ given by Lemma~\ref{lemma-smooth-Q-lc}. 
Since the sheaf $\cS_+(H)$ is one of the two exceptional objects generating~$\cQ_1(H)$, 
the image of the functor~$\Psi$} 
is contained in the component $\delta^*(\cQ_1(H) \boxtimes \Db(\bP(\vV)))$ of~\eqref{eq:hpd-sod-q}. 
Furthermore, it follows that
we can rewrite the components in the Lefschetz part of~\eqref{eq:hpd-sod-q} as
\begin{align*}
\delta^*(\cQ_1(H) \boxtimes \Db(\bP(\vV))) &= \langle \Psi(\Db(\bP(\vV))), h_Q^*(\Db(\bP(V^{\svee})))(H) \rangle, && \text{and}\\
\delta^*(\cQ_i(H) \boxtimes \Db(\bP(\vV))) &= h_Q^*(\Db(\bP(V^{\svee})))(iH), && \text{for $i \ge 2$.} 
\end{align*}
Comparing this with the decomposition~\eqref{eq:sod-hq-2}, we conclude that there is a $\bP(\vV)$-linear equivalence 
\begin{equation}
\label{Qv-Q-hpd}
\Db(Q^{\svee}) \simeq \Db(Q)^\hpd  .
\end{equation}
This equivalence is induced by the 
functor 
\begin{equation*}
\Phi = \gamma \circ \phi \circ \mu \circ \alpha_Q \colon \Db(Q^{\svee}) \to \Db(\bH(Q)), 
\end{equation*} 
where the factors are defined by~\eqref{eq:def-gamma}, \eqref{eq:functor-perf-z-cr-z2}, \eqref{eq:functor-perf-z-z2}, 
and~\eqref{eq:alpha-q}.
Note that each of the factors is a Fourier--Mukai functor, hence so is their composition $\Phi$. 
This shows that condition~\eqref{eq:hpd-prop-1} is fulfilled.

Let us describe the Fourier--Mukai kernel for $\Phi$ explicitly.
We consider the commutative diagram 
\begin{equation}
\label{eq:cd-phi}
\vcenter{\xymatrix{
\bH(Q, \vQ) \ar[r]^-{{\tilde{h}}} \ar[d]_-{{\tilde{g}}} & \vQ \ar[d]^-{g} \ar@/^2ex/[dr]^j \\ 
\bH(Q) \ar[r]^{{h_Q}} & \bP(\vV) & Z \ar[l]_-\zeta
}}
\end{equation} 
with cartesian square,
where $g \colon Q^{\svee} \to \bP(\vV)$ is the inclusion of the branch divisor of the double covering~$\zeta \colon Z \to \bP(\vV)$
and~$\bH(Q, \vQ)$ is the fiber product.
The functor~$\Phi$ is given by 
\begin{align}
\nonumber \cF & \mapsto {h_Q^*}{\left(\zeta_*\big((j_*\cF \otimes \chi) \otimes {\cS_Z^{\svee}}\big)^{\bZ/2} \right)} \otimes _{\Cl_0(\cW)} \CS \\
\nonumber
& \cong {h_Q^*}{\left(\zeta_*j_*\big(\cF \otimes j^*({\cS_Z^{\svee}} \otimes \chi)^{\bZ/2}\big) \right)} \otimes _{\Cl_0(\cW)} \CS \\
\nonumber
& \cong {h_Q^*}g_*{\left(\cF \otimes (j^*{\cS_Z^{\svee}} \otimes \chi)^{\bZ/2} \right)} \otimes _{\Cl_0(\cW)} \CS \\
\nonumber
& \cong \tilde{g}_*\tilde{h}^*{\left(\cF \otimes (j^*{\cS_Z^{\svee}} \otimes \chi)^{\bZ/2} \right)} \otimes _{\Cl_0(\cW)} \CS, 
\end{align}
where the first isomorphism is the equivariant projection formula, the second is evident, and the third is base change. 
This means that the kernel object for $\Phi$ is isomorphic to
\begin{equation}
\label{eq:ck-first}
\cE := \tilde{h}^*\Big((j^*{\cS_Z^{\svee}} \otimes \chi)^{\bZ/2}\Big) \otimes _{\Cl_0(\cW)} \tilde{g}^*\CS
\end{equation}
on $\bH(Q, \vQ)$.

To complete the proof of case~(1) of Theorem~\ref{main-theorem}, 
we will rewrite the formula for the kernel object~$\cE$ and then use it
to verify condition~\eqref{eq:hpd-prop-2}.

\subsection{Rewriting the kernel $\cE$} 
\label{step-7}
Note that by definition from~\S\ref{step-3} the bundle $j^*{\cS_Z^{\svee}}$ in~\eqref{eq:ck-first} is the cokernel of the natural map 
\begin{equation*}
(\sS_+ \oplus \sS_-) \otimes g^*\fZ_1  \to (\sS_+ \oplus \sS_-) \otimes \cO 
\end{equation*}
induced by the action of $\fZ_1 \subset \Cl_1(\cW) \subset \Cl_1(V) \otimes \cO$ on $(\sS_+ \oplus \sS_-) \otimes \cO$. 
Since this action swaps the grading and~$\fZ_1 \cong \det(\cW) \cong \cO(-H')$, 
it follows that on $\bP(\vV)$ we have two exact sequences 
\begin{align*}
&	0 \to \sS_- \otimes \cO(-H') \to \sS_+ \otimes \cO \to g_*((j^*{\cS_Z^{\svee}})^{\bZ/2}) \to 0,
\\
&0 \to \sS_+ \otimes \cO(-H') \to \sS_- \otimes \cO \to g_*((j^*{\cS_Z^{\svee}} \otimes \chi)^{\bZ/2}) \to 0 , 
\end{align*}
with the first morphisms given by the Clifford multiplication.
Comparing these sequences with~\eqref{eq:spinor-sequence-pn} for the spinor bundles $\cS'_+$ and $\cS'_-$ on $\vQ \subset \bP(\vV)$, 
we obtain isomorphisms 
\begin{equation*}
(j^*{\cS_Z^{\svee}})^{\bZ/2} \cong \cS'_-(H')
\qquad\text{and}\qquad 
(j^*{\cS_Z^{\svee}} \otimes \chi)^{\bZ/2} \cong \cS'_+(H').
\end{equation*}
Combining this with the formula~{\eqref{eq:ck-first} for~$\cE$, we find that $\cE$ can be rewritten as follows: 
\begin{equation}
\label{eq:hpd-kernel}
\cE = \tilde{h}^*\cS'_+(H') \otimes _{\Cl_0(\cW)} {\tilde{g}^*}\CS \in \Db(\bH(Q, \vQ)).
\end{equation}

To rewrite this further, we consider the space $\bH(\bP(V), \vQ) = \bH \times_{\bP(\vV)} \vQ$, 
which fits into a commutative diagram
\begin{equation*}
\xymatrix@C=5em{
\bH(Q,\vQ) \ar[r]_{{\tilde{\imath}}} \ar[d]^{{\tilde{g}}} \ar@/^4ex/[rr]^-{\tilde{h}} \ar@/_3em/[dd]_-{\tilde{\pi}} &
\bH(\bP(V),\vQ) \ar[d]_{g_\bH} \ar[r]_-{\pr_2} \ar@/^3em/[dd]^-{\pr_1} &
\vQ
\\
\bH(Q) \ar[r]^{i} \ar[d]^{\pi_Q} &
\bH \ar[d]_\pi
\\
Q \ar[r]^f
& \bP(V)
}
\end{equation*}
where the squares are cartesian by the definitions of $\bH(Q)$ and $\bH(Q, \vQ)$.
Using the projection formula and base change we compute
\begin{equation*}
\tilde{\imath}_*\cE 
\cong 
\pr_2^*(\cS'_+(H')) \otimes_{\Cl_0(\cW)} \tilde{\imath}_*{\tilde{g}^*}\CS 
\cong 
\pr_2^*(\cS'_+(H')) \otimes_{\Cl_0(\cW)} g_\bH^*(i_*\CS).
\end{equation*}
Using the resolution~\eqref{eq:ce-res} of $i_*\CS$ and taking into account that $\cS'_+ \otimes_{\Cl_0(\cW)} \Cl_1(\cW) \cong \cS'_-$ 
(which follows from~\eqref{eq:spm-tensor-cl1} and the resolutions~\eqref{eq:spinor-sequence-pn} for $\cS'_\pm$),
we obtain an exact sequence
\begin{equation*}
0 \to \pr_1^*\cO_{\bP(V)}(-H) \otimes \pr_2^*\cS'_+(H') \to \pr_1^*\cO_{\bP(V)} \otimes \pr_2^*\cS'_-(H') \to \tilde{\imath}_*\cE \to 0
\end{equation*} 
on $\bH(\bP(V),\vQ)$, where 
the first map is induced by the Clifford multiplication.
It follows that~$\cE$ is a sheaf on $\bH(Q,\vQ)$, 
which fits into an exact sequence 
\begin{equation*}
{\tilde{\pi}}^*\cO_{Q}(-H) \otimes {\tilde{h}}^*\cS'_+(H') \to {\tilde{\pi}}^*\cO_{Q} \otimes {\tilde{h}}^*\cS'_-(H') \to \cE \to 0,
\end{equation*}
where 
the first map is induced by the Clifford multiplication.
Consider the diagram
\begin{equation*}
\xymatrix{
\sS_- \otimes {\tilde{\pi}}^*\cO_{Q}(-H) \otimes {\tilde{h}}^*\cO_{\vQ} \ar[r] \ar@{->>}[d] &
\sS_+ \otimes {\tilde{\pi}}^*\cO_{Q} \otimes {\tilde{h}}^*\cO_{\vQ} \ar@{->>}[d] 
\\
{\tilde{\pi}}^*\cO_{Q}(-H) \otimes {\tilde{h}}^*\cS'_+(H') \ar[r] &
{\tilde{\pi}}^*\cO_{Q} \otimes {\tilde{h}}^*\cS'_-(H'),
}
\end{equation*}
where the vertical arrows are induced by~\eqref{eq:spinor-sequence-pn} (hence surjective),
and the horizontal arrows are induced by the Clifford multiplication (hence the diagram commutes). 
By~\eqref{eq:spinor-sequences-even} the image of 
the top horizontal arrow is the bundle $\tilde{\pi}^*\cS_+ \otimes \tilde{h}^*\cO_{\vQ}$ on $\bH(Q,\vQ)$, 
hence we obtain an exact sequence
\begin{equation*}
{\tilde{\pi}}^*\cS_+ \otimes {\tilde{h}}^*\cO_{\vQ} \to {\tilde{\pi}}^*\cO_Q \otimes {\tilde{h}}^*\cS'_-(H') \to \cE \to 0
\end{equation*}
on $\bH(Q,\vQ)$.
Therefore, we have on $Q \times \vQ$ an exact sequence 
\begin{equation}
\label{eq:resolution-delta-psi}
0 \to \cS_+ \boxtimes \cO_{\vQ} \to \cO_Q \boxtimes \cS'_-(H') \to \eps_*\cE \to 0,
\end{equation}
where $\eps \colon \bH(Q,\vQ) \to Q \times \vQ$ is the natural embedding.

\subsection{Lefschetz center}
\label{step-9}

To complete the proof of Theorem~\ref{main-theorem} in case~\eqref{HPD-quadrics-1}, 
we verify condition~\eqref{eq:hpd-prop-2} 
i.e. check that the functor~${\Phi}^* \circ {\pi_Q^*} \colon \Db(Q) \to \Db(\vQ)$ takes the 
Lefschetz center~$\cQ_0$ of~$Q$ (see Lemma~\ref{lemma-smooth-Q-lc}) to that of~$\vQ$. 
It suffices to check that ${\Phi}^* \circ {\pi_Q^*}$ takes 
the generators $\cO_Q$ and $\cS_+$ of $\cQ_0$
to some generators of the Lefschetz center of $\Db(\vQ)$.

For this, note that 
\begin{equation*}
{\pi_{Q*}} \circ {\Phi} \colon \Db(\vQ) \to \Db(Q) 
\end{equation*}
is the Fourier--Mukai functor given by the kernel $\eps_* \cE \in \Db(Q \times \vQ)$, 
so its left adjoint functor~${\Phi}^* \circ {\pi^*_Q}$ is given by the kernel $(\eps_*\cE)^{\svee} \otimes {p}^* \omega_Q$, 
where ${p} \colon Q \times \vQ \to Q$ is the projection and~$\omega_Q$ is the dualizing complex of~$Q$. 
By~\eqref{eq:resolution-delta-psi} we have a distinguished triangle 
\begin{equation*}
(\eps_*\cE)^{\svee} \otimes {p}^* \omega_Q \to \omega_Q \boxtimes \cS_-^{\prime\svee}(-H') 
\to (\cS_+^{\svee} \otimes \omega_Q) \boxtimes \cO_{\vQ} . 
\end{equation*} 
It follows that $\Phi^*(\pi_Q^*(\cQ_0))$ is the subcategory
generated by~$\cS_-^{\prime\svee}(-H')$ and~$\cO_{\vQ}$.
From~\eqref{eq:spinor-duality-even} it follows that $\cS_-^{\prime\svee}(-H') \cong \cS'_\pm$ (depending on parity of~$d$),
hence this subcategory coincides with the Lefschetz center of $\Db(\vQ)$.
This proves condition~\eqref{eq:hpd-prop-2}, and since together with condition~\eqref{eq:hpd-prop-1} proved in~\S\ref{step-6} it is equivalent to the HPD statement, this
completes the proof of case~\eqref{HPD-quadrics-1} of Theorem~\ref{main-theorem}. 

\subsection{Other types of quadrics}
\label{step-10}

We deduce the other cases of Theorem~\ref{main-theorem} 
by using a general result on the behavior of HPD under linear projection. 
This result can be phrased in simple terms by saying that linear projections on one side of HPD 
correspond to taking hyperplane sections on the other.
The following rigorous formulation of this result is a special case of~\cite[Theorem~1.1]{carocci2015homological};  
see also~\cite[Proposition~A.10 and Remark~A.12]{categorical-cones} for a quick proof in the context considered here. 

\begin{theorem}
\label{theorem-HPD-projection}
Let $\tf \colon X \to \bP(\tV)$ be a Lefschetz variety with Lefschetz center~$\cA_0$. 
Let $\tV \to V$ be a surjection with kernel $K$ such that $\tf^{-1}(\bP(K))$ is empty, 
so that the composition
\begin{equation*}
f \colon X \xrightarrow{\ \tf\ } \bP(\tV) \dashrightarrow \bP(V)
\end{equation*}
is a regular morphism providing~$X$ with the structure of a Lefschetz variety over $\bP(V)$ with Lefschetz center~$\cA_0$.
Assume that the Lefschetz variety $f \colon X \to \bP(V)$ 
is moderate. 

If $\tf^{\hpd} \colon X^{\hpd} \to \bP(\tV^{\svee})$ is a Lefschetz variety with Lefschetz center~$\cB_0$
which is HP dual to~\mbox{$\tf \colon X \to \bP(\tV)$}, then the fiber product
\begin{equation*}
f^{\hpd} \colon X^{\hpd} \times_{\bP(\tV^{\svee})} \bP(\vV) \to \bP(\vV) 
\end{equation*} 
obtained by base change along the natural embedding $\bP(\vV) \subset \bP(\tV^{\svee})$ 
has the structure of a Lefschetz 
variety with Lefschetz center the image of $\cB_0$ under the restriction functor 
\begin{equation*}
\Db(X^{\hpd}) \to \Db(X^{\hpd} \times_{\bP(\tV^{\svee})} \bP(\vV)). 
\end{equation*}
Moreover, when equipped with this Lefschetz structure, $f^{\hpd} \colon X^{\hpd} \times_{\bP(\tV^{\svee})} \bP(\vV) \to \bP(\vV)$ 
is HP dual to $f \colon X \to \bP(V)$.
\end{theorem}

Let us prove case~\eqref{HPD-quadrics-dc-even}.
Let $f \colon Q \to \bP(V)$ be a double covering of an even-dimensional quadric~$Q$.
Choose an embedding $\tilde{f} \colon Q \to \bP(\tV)$ as a hypersurface.
Then $f$ is the composition of $\tilde{f}$ with a linear projection~$\bP(\tV) \dashrightarrow \bP(V)$ 
from a point of $\bP(\tV)$ which does not lie on~$\tilde{f}(Q)$.
Let $K \subset \tV$ be the corresponding 1-dimensional subspace, so that $V = \tV/K$.

By case~\eqref{HPD-quadrics-1} of Theorem~\ref{main-theorem} proved in~\S\S\ref{step-1}--\ref{step-9}, 
we know that $\vQ \subset \bP(\tV^{\svee})$ is HP dual to~$Q \subset \bP(\tV)$.
Therefore, by Theorem~\ref{theorem-HPD-projection} we find that 
$\vQ \times_{\bP(\tV^{\svee})} \bP(\vV) \to \bP(\vV)$ is HP dual to $Q \to \bP(V)$.
Note that $\vQ \times_{\bP(\tV^{\svee})} \bP(\vV)$ is projectively dual 
to the branch divisor of~\mbox{$Q \to \bP(V)$}.
This proves case~\eqref{HPD-quadrics-dc-even} of Theorem~\ref{main-theorem}. 

Since the operation of HPD is a duality (see \cite[Theorem 7.3]{kuznetsov-hpd} or \cite[Theorem 8.9]{NCHPD}), 
this also proves case~\eqref{HPD-quadrics-d-odd} of the theorem.

Finally, let us prove case~\eqref{HPD-quadrics-dc-odd}.
Let $f \colon Q \to \bP(V)$ be a double covering of an odd-dimensional quadric~$Q$.
Choose an embedding $\tilde{f} \colon Q \to \bP(\tV)$ as a hypersurface.
Then $f$ is the composition of $\tilde{f}$ with a linear projection~$\bP(\tV) \dashrightarrow \bP(V)$ 
from a point of $\bP(\tV)$ which does not lie on~$\tilde{f}(Q)$.
Let $K \subset \tV$ be the corresponding 1-dimensional subspace, so that $V = \tV/K$.

By case~\eqref{HPD-quadrics-d-odd} of Theorem~\ref{main-theorem} proved above, 
the double cover $\cov{(\vQ)} \to \bP(\tV^{\svee})$ branched along the projective dual $\vQ \subset \bP(\tV^{\svee})$ 
is HP dual to $Q \subset \bP(\tV)$. 
Thus, using Theorem~\ref{theorem-HPD-projection} again,
we find that $\cov{(\vQ)} \times_{\bP(\tV^{\svee})} \bP(\vV) \to \bP(\vV)$ is HP dual to 
$Q \to \bP(V)$. 
Note that the variety $\vQ \times_{\bP(\tV^{\svee})} \bP(\vV)$ is projectively dual
to the branch divisor of $Q \to \bP(V)$, and 
that~$\cov{(\vQ)} \times_{\bP(\tV^{\svee})} \bP(\vV) \to \bP(\vV)$ is the double cover branched along $\vQ \times_{\bP(\tV^{\svee})} \bP(\vV)$. 
This proves case~\eqref{HPD-quadrics-dc-odd} of Theorem~\ref{main-theorem}. 
\qed 

\begin{remark}
\label{remark:hpd-kernels}
We finish the paper by noting that the HPD kernel $\cE$ in all the cases of Theorem~\ref{main-theorem} 
fits into an exact sequence
\begin{equation*}
0 \to \cS \boxtimes \cO_{\vQ} \to \cO_Q \boxtimes \cS'(H') \to \eps_*\cE \to 0,
\end{equation*}
where $\eps \colon \bH(Q,\vQ) \to Q \times \vQ$ is the embedding, 
$\cS$ and $\cS'$ are appropriate spinor bundles on~$Q$ and~$\vQ$, 
and the first morphism is the composition of (the pullback from $Q$ of) the natural embedding of $\cS$ into the half-spinor representation 
with (the pullback from $\vQ$ of) the surjection from the half-spinor representation onto $\cS'(H')$ 
(see~\eqref{eq:spinor-sequences-even} and~\eqref{eq:spinor-sequences-odd}).
Indeed, in case~\eqref{HPD-quadrics-1} this was already shown in~\eqref{eq:resolution-delta-psi}.
Further, in case~\eqref{HPD-quadrics-dc-even} the HPD kernel is obtained by restriction, 
hence~\eqref{eq:spinor-restriction} shows that the above formula is still true.
Further, in case~\eqref{HPD-quadrics-d-odd} the HPD kernel is obtained by transposition and dualization,
hence~\eqref{eq:spinor-duality-even} and~\eqref{eq:spinor-duality-odd} imply the formula.
Finally, in case~\eqref{HPD-quadrics-dc-odd} the HPD kernel is again obtained by restriction, 
so we again conclude by~\eqref{eq:spinor-restriction}.
\end{remark}


\providecommand{\bysame}{\leavevmode\hbox to3em{\hrulefill}\thinspace}
\providecommand{\MR}{\relax\ifhmode\unskip\space\fi MR }
\providecommand{\MRhref}[2]{%
  \href{http://www.ams.org/mathscinet-getitem?mr=#1}{#2}
}
\providecommand{\href}[2]{#2}


\end{document}